\documentclass[11pt]{amsart}
\usepackage{amsfonts}

\usepackage[all,cmtip]{xy}
\usepackage{amsmath}
\usepackage{amsthm}
\usepackage{amssymb}
\usepackage{enumitem}
\newtheorem{thm}{Theorem}[section]
\newtheorem{conj}{Conjecture}[section]
\newtheorem{prop}[thm]{Proposition}
\newtheorem*{definition*}         {Definition}

\newtheorem{cor}[thm]{Corollary}

\theoremstyle{remark}
\newcommand*{\Oo}{\mathcal{O}}
\newcommand*{\Q}{\mathbb{Q}}
\newcommand*{\Hh}{\mathbb{H}}
\newcommand*{\A}{\mathcal{A}}
\newcommand*{\Qa}{\overline{\mathbb{Q}}}
\newcommand*{\Z}{\mathbb{Z}}

\newcommand*{\R}{\mathbb{R}}

\newcommand*{\C}{\mathbb{C}}

\newcommand*{\Disc}{\textrm{Disc}}
\newcommand*{\Gal}{\textrm{Gal}}

\DeclareMathOperator{\spec}{Spec}
\DeclareMathOperator{\fal}{Fal}

\newcommand*{\ra}{\rightarrow}
\newcommand*{\ol}{\overline}

\def\Sp{{\rm Sp}}

\def\Disc{{\rm Disc}}
\usepackage{amscd,amssymb,amsmath}
\usepackage{color}

\author{Jacob Tsimerman}
\address{Department of Mathematics, University of Toronto, Canada}

\begin{document}

\begin{abstract}
We give a proof of the Andr\'e-Oort conjecture for $\mathcal{A}_g$ - the moduli space of principally polarized abelian varieties. In particular, we show that a recently proven `averaged' version of the Colmez conjecture yields lower bounds for Galois orbits of CM points. The Andr\'e-Oort conjecture then follows
from previous work of Pila and the author.
\end{abstract}

\title{A proof of the Andr\'e-Oort conjecture for $\mathcal{A}_g$}
\maketitle
\section{Introduction}

Recall the statement of the Andr\'e-Oort conjecture: 

\begin{conj}\label{AO}
Let $S$ be a Shimura variety, and let $V$ be an irreducible closed algebraic subvariety of $S$.
Then $V$ contains only finitely many
maximal special subvarieties.
\end{conj}

In the past few decades, there has been an enormous amount of work on the Andr\'e-Oort  conjecture. Notably, a full proof of AO
under the assumption of the Generalized Riemann Hypothesis (GRH) for CM fields
has been given by Klingler, Ullmo, and Yafaev \cite{KYAO, UYAO}.

Following a strategy introduced by Pila and Zannier, in previous work \cite{PT} it was shown that the Andr\'e-Oort conjecture for coarse the moduli space of principally polarized abelian varieties $\mathcal{A}_g$ follows once one establishes that the sizes of the Galois orbits of special points are `large'. This is what we prove in this paper. More specifically, we prove the following 

\begin{thm}\label{lowerprimtop}

Then there exists $\delta_g>0$ such if $\Phi$ be a primitive CM type for a CM field $E$, and if $A$ is any abelian variety of dimension $g$ with endomorphism ring equal to the full ring of integer $\Oo_E$ and CM type $\Phi$, then the field of moduli\footnote{By `field of moduli' here we mean the intersection of all fields over which $A$ has a model, or alternatively the field over which the point $A$ is defined in the moduli space} $\Q(A)$ of $A$ satisfies
$$[\Q(A):\Q] \gg |\Disc(E)|^{\delta_g}.$$

\end{thm}

Though it is explained in detail in \cite{PT} and other wonderful survey papers how the general strategy for Andr\'e-Oort works, in an effort to be as self-contained a possible we give a short sketch of the main story in section 6. The reader who is already familiar with this story or is only interested in the proof of Theorem \ref{lowerprimtop} may safely skip this section.

As the argument for Theorem \ref{lowerprimtop} is relatively short, let us first give a quick sketch of it.

\subsection{Proof sketch of Theorem \ref{lowerprimtop}}

Let $E$ be a CM field with totally real subfield $E_0$ and $\Phi$ be a CM type for $E$. Now let $S(E,\Phi)$ denote the set of complex abelian varieties $A$ with complex multiplication by $(\Oo_E,\Phi)$. That is, $\Oo_E$ acts on $A$, such that the induced representation of $E$ on $T_0A(\C)$ given by $\Phi$. Our starting point is the observation that everything in $S(E,\Phi)$ is mutually isogenous, and the field of moduli $K$ of all these abelian varieties is the same. Moreover, it is well known that there are at least $\Disc(E)^{1/4+o(1)}$ elements in $S(E,\Phi)$ --- we note that the exponent of $1/4$ is irrelevant for us, we only care that it is some positive constant.

Next, Colmez \cite{C} has proven that the Faltings heights of all the elements in $S(E,\phi)$ are the same. Moreover, he has conjectured a precise formula for the Faltings height. All that matters for our purposes is that his conjectural formula is subpolynomial in $\Disc(E)$. Now, while Colmez's conjecture is not yet proved, an averaged version has been announced by Andreatta-Goren-Howard-Madapusi Pera and more recently by Yuan-Zhang, and this is sufficient to establish our desired upper bound for the Faltings height.

Finally, the theorem of Masser-W\"ustholz now says that all the elements of $S(E,\phi)$ have isogenies between them of degree at most
$\newline\max(h_{\fal}, [K:\Q])^{c_g}$, where $h_{\fal}$ denotes their common Faltings height, and $c_g$ is some positive constant depending only on the degree $g$ of $E$. Since there are at most polynomially many isogenies of a given degree $N$, it follows that $[K:\Q]$ must grow at least polynomially in $\Disc(E)$. 

\subsection{Paper outline}

In Section 2 we gather the basic facts we need about Faltings heights, CM abelian varieties, and the theorem of Masser-W\"ustholz. In Section 3 we show how the average version of the Colmez conjecture yields upper bounds for Faltings heights of CM abelian varieties. In Section 4 we use these upper bounds, together with the theorem of Masser-W\"ustholz to prove our desired lower bounds on Galois orbits. In Section 5 we recall how these lower bounds imply the Andr\'e-Oort conjecture for $\A_g$. In Section 6, we give a brief sketch for the interested reader of the complete proof of Andr\'e-Oort, putting together the ingredients in the literature to explain the complete argument.

\subsection{Acknowledgements} It is a great pleasure to thank Jonathan Pila for asking the key question to inspire this work, and for pointing the author in
the direction of Colmez's conjecture, as well as several helpful discussions.
I am also grateful to Shou-Wu Zhang, Ben Howard and Stephen Kudla for helpful conversations. Thanks also to Peter Sarnak for a careful reading of an earlier version of this paper and making useful suggestions, and thanks to Bjorn Poonen, Arul Shankar and Ila Varma for carefully going through the paper and pointing out several errors and inconsistencies.

\section{Background}

\subsection{Faltings Height} Let $A$ be an abelian variety over $\Qa$. Let $K$ be any subfield of $\Qa$ such that $A$ is definable over $K$ with everywhere semi-stable reduction.
Now let $\pi:\mathcal{A}\ra\spec\Oo_K$ be the N\'eron model of $A$, and take $\omega$ to be any global section of $\mathcal{L}:=\pi_*\Omega^g_{\mathcal{A}/\spec\Oo_K}$. Then we define the Faltings height of $A$ as follows:

$$h_{\fal}(A):= \frac1{[K:\Q]}\left(\log|H^0(\spec\Oo_K,\mathcal{L}):(\Oo_K\cdot\omega)|-\frac12\sum_{\sigma:K\ra\C}\log\int_{A(\C)}\sigma(\omega\wedge\overline{\omega})\right).$$

This turns out to be a well defined quantity independent of the choice of $K$ and $\omega$. It can be thought of as a measure of the arithmetic complexity of $A$. We shall need the following theorem of Bost \cite{B}:

\begin{thm}\label{bost}
There exists a constant $c_g$ depending on $g$ alone such that for $A$ an abelian variety over $\Qa$ of dimension $g$, $h_{\fal}(A)\geq c_g$. In fact, one can take $c_g$ to be linear in $g$.
\end{thm}

\subsection{Complex Multiplication}

Let $E$ be a CM field with totally real subfield $E_0$, set $g=[E:\Q]$ and let $\Phi$ be a CM type of $E$. That is, $\Phi$ is a set of embeddings $E\hookrightarrow\C$ such that the set of all embeddings is given by the disjoint union of $\Phi$ and $\overline{\Phi}$. Next, define $S(E,\Phi)$ to be the set of isomorphism classes of $g$-dimensional complex abelian varieties $A$ together with an embedding $\Oo_E\ra \textrm{End}_{\C}(A)$ such that the induced action of $E$ on the tangent space $T_0A(\C)\cong \C^g$ is given by $\Phi$. It is known that the number of elements in $S(E,\Phi)$ is given by \cite[Propositions 7.17]{ST}:
$$|S(E,\Phi)| =|Cl(E)|$$
where $Cl(F)$ denotes the class number of a field $F$.

In general class numbers of algebraic number fields can be very small, but this is not so for CM fields.
The Brauer-Siegel theorem, the fact that the regulators of $E$ and $E_0$ are the same, and the fact that $|\Disc(E)|\geq |\Disc(E_0)|^2$ imply that 
$$|S(E,\Phi)| \gg_g \frac{|Cl(E)|}{|Cl(E_0)|}\gg_g\Disc(E)^{1/4-o_g(1)}.$$

Now, suppose that $\Phi$ is a primitive CM type, which means that it is not induced from a CM type of a strict CM subfield of $E$. Then $\textrm{End}_{\C}(A) \cong \Oo_E$ for any $A\in S(E,\Phi)$, so isogenies between abelian varieties in $S(E,\Phi)$ necessarily respect the $\Oo_E$ action. Finally, all the elements in $S(E,\Phi)$ are isogenous, and isogenies correspond to ideals in $\Oo_E$ in the following way:  given an ideal $I\subset\Oo_E$, let $T_I$ be the kernel of $I$ acting on $A$. Then the map $A\ra A/T_I$ is an isogeny between 2 elements in $S(E,\Phi)$ of degree the norm of $I$, and each isogeny arrises in this way. Since there are $n^{o(1)}$ ideals of norm $n$, we have proven

\begin{prop}\label{distinct}

For a primitive CM type $\Phi$, there are two elements in $S(E,\Phi)$ such that the lowest degree isogeny between them has degree at least $|\Disc(E)|^{1/4-o_g(1)}$.
\end{prop}
\subsection{Masser-W\"ustholz Isogeny Theorem}

We shall make heavy use of the following theorem of Masser-W\"ustholz \cite{MW}:

\begin{thm}\label{MaW}

Let $A,B$ be abelian varieties of dimension $g$ over a number field $k$, and suppose that there exists an isogeny between them over $\C$. Then if we let $N$ be the minimal degree of an isogeny between them over $\C$, we have the bound

$$N\ll_g \textrm{Max}(h_{\fal}(A), [k:\Q])^{c_g}$$ where $c_g$ is a positive constant depending only on $g$.

\end{thm}

This theorem is proved in \cite{MW}, even though its stated a little differently. See the remark on the bottom of page 23.

\section{Colmez's conjecture}

It is a theorem of Colmez \cite{C} that all the elements in $S(E,\Phi)$ have the same Faltings height denoted by $h_{\fal}(E,\Phi)$. Moreover, he conjectured a precise formula for this height as follows:

\begin{conj}\label{colmez}	

We have the identity

$$h_{\fal}(E,\Phi) = \sum_{\rho} c_{\rho,\Phi}\left(\frac{L'(0,\rho)}{L(0,\rho)} + \log f_{\rho}\right)$$ where $\rho$ ranges over irreducible complex representations of the Galois group of the normal closure of $E$ for which $L(0,\rho)$ does not vanish, $c_{\rho,\Phi}$ are rational numbers depending only on the finite combinatorial data given by $\Phi$ and the Galois group of the normal closure of $E$, and $f_{\rho}$ is the Artin conductor of $\rho$.

\end{conj}

While this conjecture is still open, a recent `averaged' version is a work in progress by Andr\'eatta, Goren, Howard and Madapusu-Pera \cite{AGHM}( for their analogue in the orthogonal case, see http://arxiv.org/abs/1504.00852), while the version without
the error term stated below has been recently announced by Yuan and Zhang \cite{YZ}. 

\begin{thm}\label{colmeza}

Colmez's conjecture holds if one averages over all CM types, up to a small error. Precisely:

$$\sum_{\Phi}h_{\fal}(E,\Phi) = \sum_{\Phi} \left(\sum_{\rho} c_{\rho,\Phi}\left(\frac{L'(0,\rho)}{L(0,\rho)} + \log f_{\rho}\right)\right) + O_g(\log|\Disc(E)|)$$ where the outer sum is over all $2^g$ CM types of $E$. Moreover, the error term can be written as a sum of rational multiples of logarithms of primes dividing $2\Disc(E)$.
\end{thm}

We shall be interested in the following corollary:

\begin{cor}\label{colmezheightbound}

We have the bound $$h_{\fal}(A)\leq|\Disc(E)|^{o_g(1)}.$$

\end{cor}

\begin{proof}

By Theorem \ref{bost} we have $$h_{\fal}(A)\leq -(2^g-1)c_g+\sum_{\Phi}h_{\fal}(E,\Phi) $$ and thus it suffices to show that 
$$\sum_{\Phi}h_{\fal}(E,\Phi) \leq|\Disc(E)|^{o_g(1)}.$$ We shall do this by showing that every term on the right hand side of theorem \ref{colmeza} is bounded above by 
$\Disc(E)|^{o_g(1)}.$

Firstly, note that for any irreducible Artin representation $\rho$ we have $f_{\rho} \leq |\Disc(E)|$, and so  $$\log f_{\rho} \leq |\Disc(E)|^{o_g(1)}. $$ Next, note that if we logarithmically differentiate the functional equation for the Artin $L$-function, we obtain that $\frac{L'(1,\rho)}{L(1,\rho)} + \frac{L'(0,\ol{\rho})}{L(0,\ol{\rho})}  = O_g(\log f_{\rho})$, and thus it suffices to bound 
$\left|\frac{L'(1,\ol{\rho})}{L(1,\ol{\rho})}\right|$

By Brauer's theorem on induced characters, every Artin L-function is a product of quotients of Hecke L-functions, and so by the Brauer-Siegel theorem we have the estimate $$L(1,\ol{\rho})= |\Disc(E)|^{o_g(1)}.$$

Finally, by Cauchy's theorem we can express $L'(1,\ol{\rho})$ as an average of $L(s,\ol{\rho})$ over an arbitrary small circle centered around $1$ in the complex plane. The standard convexity estimate for $L(s,\ol{\rho})$ \cite[(5.2)]{IK} now yields $$L'(1,\ol{\rho})\leq |\Disc(E)|^{o_g(1)}.$$ This completes the proof.
\end{proof}

%
%
%
%
%
%
%
%
%
%
%
%
%

\section{Lower bounds for Galois orbits}

We are now ready to prove Theorem \ref{lowerprimtop}:

\begin{thm}\label{lowerprim}

Then there exists $\delta_g>0$ such that if $E$ is a CM field of degree $2g$,  $\Phi$ is a primitive CM type for $E$,  and $A$ is an abelian variety in $S(E,\Phi)$, then the field of moduli $\Q(A)$ of $A$ satisfies $$[\Q(A):\Q] \gg |\Disc(E)|^{\delta_g}.$$

\end{thm}

\begin{proof}

First note that the fields of moduli of all the elements in $S(E,\Phi)$ are the same by the main theorem of complex multiplication. Next, note that there are 2 elements $A,B$ in $S(E,\Phi)$ such that the minimal isogeny between them is of degree at least $\Disc(E)^{1/4+o(1)}$ by proposition \ref{distinct}. Thus applying Theorem \ref{MaW} and Corollary \ref{colmezheightbound} we learn that $[\Q(A):\Q]^{c_g}\geq \Disc(E)^{1/4-o(1)}$, which proves the result for any $\delta_g<\frac1{4c_g}$.

\end{proof}

\section{The Andr\'e-Oort Conjecture}

Recall the statement of the Andr\'e-Oort conjecture for $\A_g$: 

\begin{conj}\label{AO}
Let $V$ be an irreducible closed algebraic subvariety of $\A_g$.
Then $V$ contains only finitely many
maximal special subvarieties.
\end{conj}

For a point $x\in\A_{g}(\overline{\Q})$ let $A_x$ denote the corresponding
$g$-dimensional
principally polarized abelian variety,
$R_x=Z({\rm End}(A_x))$ the centre of the endomorphism ring of $A_x$,
and ${\rm Disc\/}(R_x)$ its discriminant.
In general we have the following lower bound suggested by Edixhoven in
 \cite{EMO}.

\begin{thm}\label{lower}
Let $g\ge 1$. There exists a constant $b_g>0$ such that, for a special point $x\in\A_{g}$,
$$
|{\rm Disc\/}(R_x)| \ll_g |{\rm Gal}(\overline{\Q}/\Q)\cdot x|^{b_g}\
$$
(with the implied constants depending on $g$).
\end{thm}

\begin{proof}
Let $K$ be a CM field of degree $g$ with a CM type $\phi$ and let $L$ be the normal closure of $\Q$. Then there is a reciprocity homomorphism between class groups $r_{K,\phi}: Cl(L)\ra CL(K)$ defined as follow: let $\phi_L$ denote the embeddings $L\ra \ol{K}$ extending all the embeddings in $\phi$. Then if $I$ is a fractional ideal of $L$, $r_{K,\phi}(I)$ is the unique fractional ideal in $K$ satisfying
$$r_{K,\phi}(I)\Oo_L = \prod_{\sigma\in\phi_L} \sigma(I).$$ In \cite[Theorem 7.1]{JT} it is shown that Theorem \ref{lower} follows from the following purely-fields theoretic statement:
There exists a positive constant $\delta(g)$ depending only on $g$ such that for any pair $(K,\phi)$ with $\phi$ a primitive CM-type,
\begin{equation}\label{classim}
|\textrm{im}(r_{K,\phi})|\gg_{g,\epsilon} \Disc(K)^{\delta(g)-\epsilon}.
\end{equation}

Moreover, for $A\in S(K,\phi)$, by \cite[\S15, Main Theorem 1]{ST}, $\deg\Q(A)$ and $|\textrm{im}(r_{K,\phi})|$ are almost the same size. Since we cannot find an adequate reference, we explain this point in some detail:

Let $(K^*,\phi^*)$ denote the dual CM field and dual CM type of $(K,\phi)$. Now, as can be checked, $r_{K,\phi}$ descends to a map $r'_{K,\phi}:Cl(K^*)\ra CL(K)$. Deifine a subgroup $H$ of ideal classes $[I]$ in $Cl(K^*)$ such that $r'_{K,\phi}(I')= (a)$ and $|\Oo_K/I| = a\ol{a}$ for some $a\in K^{\times}$.  Then $\deg\Q(A) = |Cl(K^*)|/|H|$. 

Now, clearly $H$ is contained in the kernel of $r'_{K,\phi}$. Moreover, for $[I]$ in the kernel of $r'_{K,\phi}$, we can find $a\in K^{\times}$ such that $r'_{K,\phi}(I')= (a)$.
This implies that $Nm(I) = r'_{K,\phi}(I)\ol{r'_{K,\phi}(I)} = a\ol{a}\Oo_K$, so that $\frac{a\ol{a}}{|\Oo_K/I|}$ is a totally positive unit of $K$. Moreover, $a$ is well defined up to an element of the units $U_K$, and so $\frac{a\ol{a}}{|\Oo_K/I|}$ is well defined up to an element of $U_K^+/N(U_K)$, where $U_K^+$ denote the totally positive units and $N(z)=z\ol{z}$. Therefore, $H$ is the kernel of the map from $\ker r'_{K,\phi}$ to
$U_K^+/N(U_K)$. This latter group is a quotient of $U_K^+/(U_K^+)^2$, which is smaller than $2^{g-1}$ times the number of roots of unity in $K$, and thus has size bounded purely in terms of $g$. Thus, we learn that the index of $H$ in the kernel of $r'_{K,\phi}$ is $O_g(1)$, and so
$$\deg\Q(A) = |Cl(K^*)|/|H| \ll_g |Cl(K^*)|/|\ker r'_{K,\phi}| = |\textrm{im}(r'_{K,\phi})|.$$

The theorem now follows from Theorem \ref{lowerprimtop} combined with the fact that $S(K,\phi)$ is not empty.

\end{proof}

In \cite[Theorem 7.1]{PT} it is proved that Theorem \ref{lower} implies the Andr\'e-Oort conjecture for $\mathcal{A}_g$. Thus, we obtain the following:

\begin{thm}\label{AOAG}

The Andr\'e-Oort Conjecture holds for $\mathcal{A}_g$ for any $g\geq 1$.
\end{thm}

Finally, we point out that Ziyang Gao has recently shown \cite{G} that the Andr\'e-Oort conjecture for any mixed Shimura variety whose pure part is a 
Shimura subvariety of $\A_g$ follows from conjecture \ref{lower}, and thus is a consequence of theorem \ref{lowerprim}.

\section{Sketch of the complete proof of Andr\'e-Oort}

In this section, for the readers convenience we outline the proof of Theorem \ref{AOAG} from `first principles', quoting all the big results we need from the literature, pointing out the use of the main ingredients \cite{PT,PW} and Theorem \ref{lowerprimtop}.


\subsection{Setup}

Since CM points are algebraic, by replacing $V$ with the irreducible components of the Zariski closure of its CM points we may assume that $V$ is algebraic as well. Without loss of generality, $V$ is defined over some number field $L$. We adopt the notation $F\prec G$ if there exist constants $A,B>0$ such that $F\leq A\cdot G^B$.

We begin with the observation that $\A_g$ has a uniformization by the Siegel upper half plane $\Hh_g$ of symmetric $g\times g$ matrices whose imaginary part is positive definite. Namely, there is a covering map $\pi_g:\Hh_g\ra \A_g$, and an action of $\Gamma_g:=\Sp_{2g}(\Z)$ on $\Hh_g$ such that $\pi_g$ is invariant under the action of 
$\Gamma_g$ and induces an isomorphism $\pi_g:\Gamma_g\backslash\Hh_g\ra \A_g$. Now $\Hh_g$ is a symmetric space and there is a well-known fundamental domain $F_g\subset \Hh_g$ for the action of $\Gamma_g$ with nice properties, so that $\pi_g$ induces a homeomorphism from $F_g$ to an open dense subset of $\A_g$. 

\subsection{ Getting polynomially many points of small height in $\pi_g^{-1}(V)\cap F_g$}$\newline$

Now, even though the map $\pi_g$ is highly transcendental, the pullback under $\pi_g\mid_{F_g}$ of a CM point $x$ of $\A_g$ is an algebraic point $y$ of degree bounded by $2g$. Moreover, this pullback isn't `too big', in the sense that the naive height $H(y)$ is polynomially bounded in $|{\rm Disc}{R_x}|$ (see Theorem 3.1 in \cite{PT2}). Now take a CM point $x\in V$  and set  $X=H(y)$. Then the orbit of $x$ under the Galois group $\Gal(\Qa/L)$ is also in $V$. Now, since $L$ is fixed, we have
$$|\Gal(\Qa/L)\cdot x|\gg|\Gal(\Qa/\Q)\cdot x|.$$ Thus we have $$H(y) \prec |\rm{ Disc}{R_x}|\prec \Gal(\Qa/L)\cdot x $$ where the 2nd inequality follows from theorem \ref{lower}.  Finally, note that if $x'$ and $x$ are in the same Galois orbit then $R_{x'}\cong R_x$.

Thus, for any set $S$ letting $N(S,X)$ be the number of degree $2g$ algebraic points of $S$ whose height is at most $X$, we have shown that 
there are arbitrarily large $X$ such that $N(\pi_g^{-1}(V)\cap F_g,X)$ is bounded below by a fixed power of $X$.

\subsection{Getting an algebraic subvariety in $\pi_g^{-1}(V)\cap F_g$}$\newline$

Now, while $\pi_g^{-1}(V)\cap F_g$ is not an algebraic or even semialgebraic set, it turns out that it can be defined using subanalytic functions, together with the exponential function. That is to say, it is in the structure $\R_{an,exp}$ \cite{PEST}. This is in some sense a very weak property of a set, but it turns out to be surprisingly useful. The key is that $\R_{an,exp}$ is an o-minimal structure. We will not say more on this topic here, and the interested reader should see \cite{DM}. We shall only need the following extremely powerful theorem of Pila-Wilkie\cite{PW}: 

\begin{thm}

For any set $T$ in $\R_{an,exp}$, let $T^{alg}$ denote the union of all connected, positive dimensional semi-algebraic sets in $T$. Then $N(T\backslash T^{alg},X)$ grows subpolynomially in $X$. 

\end{thm}

Combining the above theorem with the result of the previous section, it follows that $\pi_g^{-1}(V)\cap F_g$ contains a semi-algebraic set $W$.

\subsection{Getting a special subvariety in $V$}$\newline$

We are now in the strange situation that the pullback of an algebraic set $\pi_g^{-1}(V)$ contains a semi-algebraic set $W$, even though $\pi_g$ is transcendental. It turns out that this canonly happen if a special subvariety is involved. Formally, $V$ contains a special subvariety $S$ such that $\pi_g^{-1}(S)$ contains $W$. This is known as the hyperbolic-Ax-Lindemann theorem, and is the main result of \cite[Theorem 6.1]{PT}.

\subsection{Finishing up} 

If one goes through the above arguments carefully, we have proved that all but a finite number of the CM points on $V$ lie on a positive-dimensional special subvariety of $V$. Through slightly more complex but similar arguments, one can show that all but a finite number of the special subvarieties of $V$ lie on higher dimensional special subvarieties of $V$. Theorem \ref{AOAG} follows easily from this.

\end{document}